\documentclass{amsart}
\usepackage{amsmath, amsfonts, amssymb,amsthm}

\newtheorem{theorem}{Theorem}

\newtheorem{cor}[theorem]{Corollary}

\newtheorem*{MainTheorem}{Main Theorem}
\newtheorem*{RuTheorem}{Ru's Theorem}
\newtheorem*{FirstMainTheorem}{First Main Theorem}

\theoremstyle{definition}

\def\PP{\mathbb P}

\def\ord{\text{ord}}
\def\K{{\bf K}}

\begin{document}
\title[Non-Archimedean Defect Relation]
{A defect relation for non-Archimedean analytic curves in
arbitrary projective varieties}
\author{Ta Thi Hoai An}
\address{Institute of Mathematics\\
18 Hoang Quoc Viet, Cau Giay\\
Hanoi, VietNam}
\email{tthan@math.ac.vn}
\thanks{Financial support provided to the author as a Junior Associate
by ICTP, Trieste, Italy and by SIDA} 
\begin{abstract}
If $f$ is a non-Archimedean analytic curve in a
projective variety $X$ embedded in $\PP^N$
and if $D_1,\dots,D_q$ are hypersurfaces of $\PP^N$ in general position
with $X,$ then we prove the defect relation:
$$
    \sum_{j=1}^q \delta(f,D_j) \le \dim X.
$$
\end{abstract}
\maketitle


\section{introduction} Nevanlinna theory is a generalization of the
fundamental theorem of algebra to meromorphic functions over the
complex numbers, or more generally to meromorphic maps between
complex spaces.  In the past few decades, there has been interest
in a non-Archimedean Nevanlinna theory, \textit{cf.,} \cite{HY}.

Classical Nevanlinna theory for meromorphic functions
over the complex numbers
consists of two fundamental
inequalities.  The first main theorem, which is a reformulation
of the Poisson-Jensen formula,  provides an upper bound
on the frequency with which a meromorphic function can take on
a complex value in terms of the growth order of the function.
The deeper second main theorem provides a lower bound.
At the same time, the second main theorem gives an upper-bound
on the ramification, \textit{i.e.,} the frequency that the function can take
on values with multiplicity greater than one.
The first main theorem can easily be extended to the general case
of meromorphic mappings between complex spaces.
However, analogs of the second main theorem remain largely
conjectures when the target is something other than a projective
space or a semi-Abelian variety.

In \cite{R}, Ru observed that in stark contrast to the situation
over the complex numbers, over non-Archimedean ground fields,
if one drops the ramification term from the second main theorem,
then it follows from the first main theorem.  Using this observation,
Ru gave a simple proof of an inequality of second main theorem type,
without ramification, for
curvilinear hypersurfaces in general position in projective space.
Note that Ru states his results using the concrete ground field
$\mathbb{C}_p,$ the $p$-adic complex numbers.

In this note, we extend Ru's result to arbitrary subvarieties
of projective space and remark that the argument works over any
algebraically closed field complete with respect to a non-Archimedean
absolute value, even in positive characteristic.

\section{Notation and Statement of Results}
Let ${\bf K}$ be an algebraically closed field of arbitrary characteristic,
complete with respect to a non-Archimedean absolute value $|~|.$
If
$$
    f(z)=\sum_{j=0}^\infty a_n z^n
$$
is an entire function on $\K,$ then for each real number $r\ge0,$
we define
$$
    |f|_r = \sup_j |a_j|r^j = \sup \{|f(z)| : z \in \K
    \textnormal{~with~} |z|\le r\}
    = \sup \{|f(z)| : z \in \K \textnormal{~with~} |z|=r\}.
$$
Let $f:\K\to \PP^N(\K)$ be a non-constant
analytic curve in projective space. Let
$\tilde{f}=(f_0,...,f_N)$ be a reduced representative of $f$,
where $f_0,...,f_N$ are entire functions on $\K$ without
common zeros, at least one of which is non-constant.

The Nevanlinna characteristic function $T_f(r)$ is
defined by $T_f(r)=\log ||f||_r,$ where
$$
    ||f||_r=\max\{|f_0|_r,...,|f_N|_r\}.
$$
The above definition of $T_f(r)$ is independent, up to an additive
constant, of the choice of the reduced representation of $f$.

Let $D$ be a
hypersurface in $\PP^N(\K)$ of degree $d$. Let $Q$ be the
homogeneous polynomial (form) in $N+1$ variables with coefficients
in $\K$ defining $D$. We consider the entire function $Q\circ
f=Q(f_0,...,f_N)$ on $\K$. Let $n_f(r,Q)$ be the number of zeros
of $Q\circ f$  in the (closed) disk $\mathbf{B}[r]=\{z \in \K : |z| \le r\},$
counting multiplicity. If
$Q\circ f\not\equiv 0$, set
\begin{eqnarray*}
N_f(r,D) = N_f(r,Q)&=&\int_0^r\frac{n_f(t,Q)-n_f(0,Q)}{t}dt+n_f(0,Q)\log r;\\
m_f(r,D) = m_f(r,Q)&=&\log\frac{||f||^d_r}{|Q\circ f|_r}.
\end{eqnarray*}
Note that up to a constant term, $m_f(r,D)$ is independent of the
choice of defining form $Q.$
The functions $N_f(r,D)$ and $m_f(r,D)$ are referred to as the counting
and proximity functions, respectively.

A collection of $q\ge N+1$ hypersurfaces $D_1,\dots,D_q$ with
associated defining forms $Q_1,\dots,Q_q,$ is said to be
\textit{in general position} in $\PP^N(\K)$ if for any subset
$\{i_0,\dots,i_N\}$ of $\{1,\dots,q\}$ of cardinality $N+1,$ $\{x
\in \PP^N(\K) : Q_{i_j}(x)=0,\; j=0,\dots,N\}=\emptyset.$
Ru~\cite{R} proved the following

\begin{RuTheorem} Let $D_1,\dots,D_q$ be hypersurfaces
of $\PP^N(\K)$ in general position. Let $f$ be a non-constant
analytic curve from $\K$ to $\PP^N(\K)$ whose image is not
completely contained in any of the hypersurfaces $D_1,\dots,D_q.$
Then for any positive real number~$r,$
$$
    \sum_{j=1}^q \frac{m_f(r,D_j)}{\deg D_j} \le NT_f(r)+O(1),
$$
where $O(1)$ is a constant independent of $r.$
\end{RuTheorem}

In this note, we show that if the image of $f$ is contained in an
$n$-dimensional subvariety $X,$ then the $N$ on the right hand
side of the inequality in Ru's theorem can be replaced by the
dimension of $X.$  Before stating the theorem, we need to
generalize the definition of ``in general position''.

Let $X$ be an $n$-dimensional (not necessarily smooth) projective
subvariety of $\PP^N(\K).$  A collection of $q\ge n+1$ hypersurfaces
$D_1,\dots,D_q$ in $\PP^N(\K)$ is said to be
\textit{in general position with $X$}
if for any subset
$\{i_0,\dots,i_n\}$ of $\{1,\dots,q\}$ of cardinality $n+1,$
$\{x \in X : Q_{i_j}(x)=0,\; j=0,\dots,n\}=\emptyset.$

\begin{MainTheorem}
Let $X\subset \PP^N(\K)$
be a projective sub-variety of dimension $n\ge 1$ over $\K$.
Let $D_1, .., D_q$ be hypersurfaces in  $\PP^N(\K)$
in general position with $X$. Let $f:\K\to X$ be a
non-constant analytic map whose image is not completely contained
in any of the hypersurfaces $D_1,\dots,D_q.$
Then, for any positive real number $r,$
$$\sum_{j=1}^q\frac{m_f(r, D_j)}{\deg D_j}\leq nT_f(r) +O(1),$$
where $O(1)$ is a constant independent of $r$.
\end{MainTheorem}

We remark that the hypersurfaces $D_j$ are hypersurfaces of the
ambient projective space $\PP^N(\K)$ and not arbitrary effective
divisors of $X.$  In more intrinsic terms, we are essentially
considering very ample effective divisors in general position on
$X.$

We note that Evertse and Ferretti \cite{EF} have proven
    a height inequality that is the arithmetic analog of the inequality
    in our main theorem here.

Define the \textit{defect} $\delta_f(D_j)$ of the hypersurface $D_j$
with respect to the map $f$ by
$$
    \delta_f(D_j) =
    \liminf_{r\to\infty}\frac{m_f(r,D_j)}{(\deg D_j) T_f(r)}.
$$
With the defect notation, our main theorem immediately results in
the following corollary.
\begin{cor}\label{cor1} With notation and assumptions as in the theorem,
$$
    \sum_{j=1}^q\delta_f(D_j)\le n=\dim X.
$$
\end{cor}

William Cherry has suggested to us that we can use the idea
of \cite[Theorem~6.1]{CY} to strengthen this to
\begin{cor}\label{cor2} With the assumptions and notation as in
the theorem, the number of hypersurfaces $D_j$ such
that $\delta_f(D_j)>0$ is at most $n.$
\end{cor}

\section{Proof of the main theorem}

We first recall the Nevanlinna  first main theorem in
non-Archimedean fields (\textit{cf.,}~\cite{R}).

\begin{FirstMainTheorem} Let $f:\K\to
\PP^N(\K)$ be an analytic curve and  let $Q$ be a homogeneous form
of degree $d$. If $Q(f)\not\equiv 0$, then for every positive real number
$r,$
$$
    m_f(r,Q)+N_f(r,Q)=dT_f(r)+O(1),
$$
where $O(1)$ is a constant independent of $r$.
\end{FirstMainTheorem}

\begin{proof}[Proof of the Main Theorem]
For $i=1,\dots,q,$ let $Q_i$ be the homogeneous form in $\K[x_0,...,x_N]$ of
degree $d_i$ defining the hypersurface $D_i.$
Let
$\Delta$ be the lest common multiple of the $d_i$'s.
Then $G_i:=Q_i^{\Delta/d_i}$ are
homogeneous polynomials of the same degree $\Delta$.

For fixed $r,$ by rearranging the
indices if necessary, we may assume that
\begin{equation}
\label{indexorder}
    |G_1\circ f|_r\le |G_2 \circ f|_r\le\cdots\le|G_q \circ f|_r.
\end{equation}

Since $Q_1,...,Q_q$ are in general position with $X$ by assumption,
$G_1,...,G_q$ are also in general position with $X$. Therefore,
$$X\cap\{G_1=0, ..., G_{n+1}=0\}=\emptyset.$$
Applying Hilbert's Nullstellensatz (\textit{cf.,} \cite{W})
to the ideal generated by the forms defining $X$ and $G_1,\dots G_{n+1},$
we see that for any integer $k, 0\le k\le
N$, there is an integer $m_k\ge \Delta$ such that
$$x_k^{m_k}=\sum_{i=1}^{n+1}b_{ik}(x_0,...,x_N)G_i(x_0,...,x_N)
\qquad\textnormal{on~}X,$$
where $b_{ik},$ for $1\le i\le N, 0\le k\le N,$ are
homogeneous forms with coefficients in $\K$ of degree
$m_k-\Delta$. Hence
$$|f_k|_r^{m_k}\le C||f||_r^{m_k-\Delta}\max \{|G_1\circ f|_r, ...,
|G_{n+1}\circ f|_r\},$$
where $C$ is a positive constant that depends only on the coefficients
of the $b_{ik}.$
Therefore,
\begin{equation}
\label{fDelta}
    ||f||_r^{\Delta}\le C\max\{|G_1\circ f|_r, ..., |G_{n+1}\circ f|_r\}.
\end{equation}

Thus, continuing to keep $r$ fixed,
\begin{eqnarray*}
\Delta\sum_{j=1}^q\frac{m_f(r,Q_j)}{\deg Q_j}
\;=\;\sum_{j=1}^qm_f(r, G_j)&=&\sum_{j=1}^q \log
\frac{||f||_r^{\Delta}}{|G_j\circ f|_r}\\
&=&\sum_{j=1}^n \log
\frac{||f||_r^{\Delta}}{|G_j\circ f|_r}
+\sum_{j=n+1}^q\frac{||f||_r^{\Delta}}{|G_j\circ f|_r}\\
\textnormal{[from~(\ref{fDelta}) and~(\ref{indexorder})]}
&\le& \sum_{j=1}^n \log
\frac{||f||_r^{\Delta}}{|G_j\circ f|_r}
+ O(1)\\
&=&\sum_{j=1}^n m_f(r,G_j) + O(1) \\
\textnormal{[By the First Main Theorem]} &\le& n\Delta T_f(r) +
O(1).
\end{eqnarray*}
Because the right hand side of the inequality no longer depends on
the arrangement of the indices in~(\ref{indexorder}) and the $O(1)$ term
is independent of $r,$ dividing both sides through by
$\Delta$ gives the theorem for all positive $r.$
\end{proof}

\begin{proof}[Proof of Corollary~\ref{cor2}]
If the corollary were not true, then $n+1$ of the hypersurfaces
could have non-zero defect.  Thus, without loss of generality,
we may suppose $\delta_f(D_j)>0$
for $j=1,2,\dots,n+1.$  By definition, that means there exists
$\varepsilon>0$ such that for all
$r$ sufficiently large,
$$
    m_f(r,D_j) > \varepsilon(\deg D_j)T_f(r)\quad
    \textnormal{for~}j=1,\dots,n+1.
$$
Letting $G_j$ be as in the proof of our main theorem, we then have
for all $r$ sufficiently large,
$$
    m_f(r,G_j) > \varepsilon\Delta T_f(r)\quad
    \textnormal{for~}j=1,\dots,n+1.
$$
Hence, for all $r$ sufficiently large and for $j=1,\dots,n+1,$
$$
    |G_j\circ f|_r < ||f||_r^{\Delta(1-\varepsilon)}.
$$
Because $f$ is non-constant, this contradicts~(\ref{fDelta})
by applying the Nullstellensatz argument of the previous proof
to $G_1,\dots,G_{n+1}.$
\end{proof}

\section{Two Variations}
Using the slicing technique of \cite{CY}, we immediately get
\begin{theorem}Let $X\subset \PP^N(\K)$
be a projective sub-variety of dimension $n\ge 1$ over $\K$.
Let $D_1, .., D_q$ be hypersurfaces in  $\PP^N(\K)$
in general position with $X$. Let $f:\K^m\to X$ be a
non-constant analytic map whose image is not completely contained
in any of the hypersurfaces $D_1,\dots,D_q.$
Then, for any positive real number $r,$
$$\sum_{j=1}^q\frac{m_f(r, D_j)}{\deg D_j}\leq nT_f(r) +O(1),$$
where $O(1)$ is a constant independent of $r$.
\end{theorem}

In \cite{C}, Cherry found it interesting to look at maps from
$\K\setminus\{0\},$ rather than from all of $\K.$
An analytic function $f(z)$ on $\K\setminus\{0\}$ is represented by
a Laurent series
$$
    \sum_{j=-\infty}^\infty a_j z^j,
$$
and for $r>0,$ we define
$$
    |f|_r = \sup_j |a_j| r^j
    = \sup \{|f(z)| : z \in \K \textnormal{~with~} |z|=r\}.
$$
A non-constant analytic curve $f:\K\setminus\{0\}\rightarrow\PP^N(\K)$
is defined by a reduced representation $(f_0,\dots,f_N),$
where $f_j$ are analytic on $\K\setminus\{0\},$ and where at least
one of the $f_j$ is not constant.  As before, set
$$
    ||f||_r = \max\{|f_0|_r,\dots,|f_N|_r\}.
$$
However, note  this time that changing the reduced representation
of $f$ will change $||f||_r$ possibly by a power of $r$ and a
multiplicative constant.

For the Nevanlinna functions, we only consider $r\ge1.$
For the characteristic function, set
$$
    T_f(r)=\log ||f||_r + \log ||f||_{r^{-1}} - 2\log ||f||_1,
$$
which is independent of the choice of reduced representation for $f.$
For $Q$ a homogeneous form defining a hypersurface $D$ of degree $d,$
define the counting function by
$$
    N_f(r,D)=N_f(r,Q) = \sum_{r^{-1}\le |z| \le r}
    \ord_z(Q\circ f)(\log r -|\log z|),
$$
and the proximity function by
$$
    m_f(r,D)=m_f(r,Q) = \log\frac{||f||_r^d}{|Q\circ f|_r}
    + \log\frac{||f||_{r^{-1}}^d}{|Q\circ f|_{r^{-1}}}
    -2\log\frac{||f||_1^d}{|Q\circ f|_1}.
$$

\begin{theorem}Let $X\subset \PP^N(\K)$
be a projective sub-variety of dimension $n\ge 1$ over $\K$.
Let $D_1, .., D_q$ be hypersurfaces in  $\PP^N(\K)$
in general position with $X$. Let $f:\K\setminus\{0\}\to X$ be a
non-constant analytic map whose image is not completely contained
in any of the hypersurfaces $D_1,\dots,D_q.$
Then, for any real number $r>1,$
$$\sum_{j=1}^q\frac{m_f(r, D_j)}{\deg D_j}\leq 2nT_f(r) +O(1),$$
where $O(1)$ is a constant independent of $r$.
\end{theorem}

\begin{proof} The proof is essentially the same as for the main theorem.
The only difference is that now we are working simultaneously
on the radii $r$ and $r^{-1}.$  In the proof of the main theorem,
we arranged $|G_j\circ f|_r$ in increasing order in~(\ref{indexorder}).
We may not be able
to use the same index order for $|G_j\circ f|_{r^{-1}}.$  Thus, in the
worst case, we may need to use $2n$ distinct proximity functions,
and hence we get $2n$ instead of $n$ on the right hand side of the
inequality.

For the first main theorem in this context, see \cite[Theorem~III.4.7]{C}.
\end{proof}

\section{Concluding Remarks}
The first main theorem and the fact that the proximity functions
are always bounded below imply that if $f$ is a non-constant
analytic curve not completely contained in a projective
hypersurface $D,$ then
$$
    0\le \delta_f(D) \le 1.
$$
Moreover, $\delta_f(D)=1$ if the image of $f$ omits $D.$ Thus,
Corollary~\ref{cor1} implies that a non-Archimedean analytic curve
in a projective variety $X$ can omit at most $\dim X$ very ample
divisors in general position.

When $X=\PP^N(\K)$ and the $D_j$ are hyperplanes in general
position, simple examples show the main results here are optimal,
even if $f$ is assumed linear non-degenerate, \textit{cf.,}
\cite{CherryYeTrans}.  In the complex case, one conjectures, \textit{cf.,}
\cite[Conjecture VII.5.1]{L}, that for $X$ smooth, for $D_j$
smooth with at worst normal crossing singularities at their
intersections, and for algebraically non-degenerate holomorphic
curves, one need not divide by $\deg D_j$ on the left hand side of
the inequality, and in place of $\dim X$ on the right hand side,
one has the degree of the anti-canonical bundle with respect to an
ample divisor.  Note that when $X=\PP^N(\mathbb{C}),$ one has
$N+1,$ the degree of the anti-canonical bundle, rather than $N,$
on the right hand side of the optimal inequality.

In light of the above cited conjecture over the complex numbers,
it would also be interesting if one could prove an inequality
involving the degree of the canonical bundle of $X$ (over
characteristic zero ground fields) or where one does not need to
divide by the degree of the hypersurfaces on the left hand side.

It might also be interesting to develop a Nevanlinna theory from
domains of the form $\K\setminus\{a_1,\dots,a_m\},$ generalizing
the $\K\setminus\{0\}$ case treated here.  One might also want to
distinguish between essential and non-essential singularities at
each puncture.

Finally, it would be interesting to obtain an inequality of second
main theorem type for non-Archimedean analytic curves into
arbitrary varieties that contains some sort of ramification term,
such as by using truncated counting functions.  We plan to address
this in a future work.

\bibliographystyle{amsalpha}

\end{document}